\documentclass[a4paper,reqno,12pt]{amsart}
\usepackage[T2A]{fontenc}
\usepackage{amsfonts, amssymb, amsmath}
\theoremstyle{plain}

\newtheorem{lemma}{Lemma}
\newtheorem{theorem}{Theorem}
\theoremstyle{remark}

\theoremstyle{definition}
\newtheorem{example}{Example}
\newtheorem{question}{Question}
\newtheorem{definition}{Definition}
\renewcommand{\le}{\leqslant}
\renewcommand{\ge}{\geqslant}

\DeclareMathOperator {\ord}{ord}
\DeclareMathOperator {\Card}{Card}
\DeclareMathOperator {\rk}{rk}

\def\u{\bold {u}}
\newcommand{\N}{\mathbb{N}}

\newcommand{\Z}{\mathbb{Z}}
\newcommand{\K}{\mathbb{K}}
\def\deg{\operatorname{deg}}
\def\ord{\operatorname{ord}}
\def\C{\mathbb C}

\def\F{\mathcal {F}}

\begin{document}
\title[]{B\'ezout theorem\\
for a graded ideal in a ring   of generalized polynomials.
} 
\author{M. V. Kondratieva} 
\address{Moscow State University\\ 
Department of Mechanics and Mathematics\\ Leninskie Gory, Moscow, 
Russia, 119991.}
\email{kondratieva@sumail.ru} 
 
\begin{abstract} 

The article proved the upper
bound of leading  coefficient of
characteristic polynomial of
graded ideal in a ring of generalized polynomials.

Examples of such rings are as well the rings of commutative polynomials
(for which the classical B\'ezout theorem holds),
 as some rings of differential operators.
For a system of  generalized homogeneous equations 
in small codimensions we   obtain
exact polynomial in $d$ estimates.
In the general case, the estimate is double exponential in
$\tau$:
$O(d^{2^{\tau-1}})$, 
where  $d$ is a maximal degree of generators 
of a graded ideal, 
 $\tau$ is it's codimension.

For systems of linear differential equations
bounds of the same asymptotics, but by other methods, were obtained
by D.Grigoryev 
in~\cite{Grig}.

\end{abstract}

\maketitle
\thispagestyle{empty}

\section{Introduction}
In algebraic geometry and commutative algebra
many studies are devoted to the Hilbert polynomial.
Differential dimension polynomial was introduced  by
E. Kolchin~\cite {Kolchin}  and 
has the
same important role in differential algebra.  
Estimations of its coefficients
are the classic unsolved problems of differential algebra.

In recent years  interest in computer algebra has increased, 
and one of its directions is the study of Gr\"obner bases.
For polynomial ideals, this notion 
was studied    quite fully, in particular
upper and lower degree bounds of
polynomials in the Gr\"obner basis by
degrees of generators of the ideal 
(see, for example, \cite{Dube}) are found. It is interesting, that the 
complexity  of Gr\"obner basis computation
 and of task of finding the leading coefficient of  
Hilbert polynomial
has different asymptotic order.

For rings of differential operators over a field, great success
in such studies  were  reached by D.Grigoryev and A. Chistov 
(see \cite {Grig}, \cite {Janet}).
Here the situation is another, 
than for polynomial ideals, and is known only
upper double exponential bound (as for 
Gr\"obner basis orders, and for the leading coefficient of dimension
polynomial).
 
To generalize these results, we will consider  rings,
introduced in
\cite {KLMP}, see definition 4.1.4.
Gr\"obner basis technique works for ideals of such 
(in the general case, non-commutative)
rings, and the concept of
characteristic
polynomial  is defined.
 
Our bound (Theorem~\ref{prim})
in  a situation when  degree of  characteristic polynomial
1 less than the maximum possible
(we will use the term codimension 1)
coincides with Kolchin’s result (see \cite{Kolchin}, p. 199).

By analogy with this linear estimate, Kolchin believed that in
other codimension $\tau$ the bound of the leading coefficient
also is
 a polynomial of degree $O(\tau)$.
In the general case, this has not yet been disproved.

In codimensions 2 for a differential dimensional polynomial
the bound is proved (see 5.6.7, \cite {KLMP}),   and it
coincides with that obtained in this paper for homogeneous systems
generalized polynomials.

Now in codimensions 3, 4, and 5 (see Theorem~\ref{prim2})
exact upper bounds are obtained for the first time.
Note (see the example~\ref{ex}) that in codimension 3
the bound  is achieved.

In the case when the characteristic polynomial of a homogeneous ideal
in the ring of generalized polynomials is a constant, we
obtain the upper
double exponential bound  of it
 by the number of generalized unknowns. This result is 
similar to the result of D. Grigoriev.

\section{Preliminary facts.}

One can find 
basic concepts and facts  in~~\cite{Kolchin,Ritt,KLMP}. 

Denote   the set  of integers by $\Z$, 
non-negative integers by $\N_0$  and   binomial coefficients 
$\frac{s(s-1)\dots(s-m+1)}{m!}$ by $\binom{s}{m}$.

For vector $e=(j_1,\dots, j_m) \in\N_0^m$, the {\bf order} of  $e$ 
is defined by 
$\ord e=\sum_{k=1}^mj_k$. 
Note that any numerical 
polynomial $v(s)$ can be written as $v(s) =\sum_i^da_i\binom{s+i}{i}$,
where $a_i\in\Z$. We call numbers $(a_d,\dots, a_0)$ {\bf
standard coefficients } of polynomial $v(s)$.

\begin{definition}\label{min}(see\cite{KLMP}, definition 2.4.9 or \cite{Kondr}).
Let $\omega=\omega(s)$ be a univariate numerical polynomial in $s$ and let
$d=\deg\omega$. The sequence of {\bf minimizing coefficients}  
$b(\omega)$ 
is the vector $b(\omega)=(b_d,\dots ,b_0)\in\Z^{d+1}$ defined by 
induction on
$d$ as follows. If $d=0$ (i.e. $\omega$ is a constant), then
$b(\omega)=(\omega)$. Let $d>0$ and $\omega(s)=\sum^d_{i=0}a_i\binom{s+i}{i}$.
Let $v(s) = \omega(s+a_d) - \binom{s+1+d+a_d}{d+1} + \binom{s+d+1}{d+1}$. Since
$\deg v<d$, one may suppose that the sequence of minimizing coefficients
$b(v)=(b_k,\dots ,b_0)$ $(0\leq k<d)$ of the polynomial $v(s)$ has been 
defined. To define the same for $\omega$ we set $b(\omega)=(a_d,0,\dots
,0,b_k,\dots ,b_0)\in\Z^{d+1}$.
\end{definition}

Now  we define the Kolchin dimension polynomial
of a subset $E\subset \N_0^m$.
\begin{definition}\label{Kolchin}
Regard the following partial order on $\N_0^m$: the relation
$(i_1,\dots, i_m)\le(j_1,\dots, j_m)$ is equivalent to
$i_k\le j_k$ for all $k=1,\dots, m$.
We  consider a function $\omega_E(s)$, that   in a point $s$ 
equals $\Card V_E(s)$, where $V_E(s)$ is the set
of points $x\in\N_0^m$ such that $\ord x\le s$ 
and for every $e\in E$ the condition $e\le x$ isn't true.
Then (see for example, \cite{Kolchin}, p.115, or \cite{KLMP}, 
theorem 5.4.1) function $\omega_E(s)$ for all sufficiently large 
$s$ is a numerical polynomial. We call this polynomial
the Kolchin {\bf dimension polynomial} of a subset $E$.
\end{definition}

Not every numerical polynomial is a Kolchin dimensional polynomial for 
some set $E$.
The connection of these concepts is established in the following theorem.

\begin{theorem}\label{w}
(see \cite{Kondr} and \cite{KLMP}, proposition 2.4.10). 
The sequence of
 minimizing  coefficients
of
dimension polynomial Kolchin consists of only non-negative integers.
The converse is also true: if the sequence of minimizing coefficients 
of some numerical  polynomial consists of non-negative numbers, 
then it is the Kolchin dimension polynomial of some set $ E $.
We denote the set of such polynomials by $W$.
\end{theorem}

Note that the set $W$ is closed with respect to addition,
difference
\begin{equation}\label{delta}
\Delta_1\omega(s)=\omega(s)-\omega(s-1)
\end{equation}
and  positive shift: ($\omega(s)=\omega(s+j), \ j\in\N$).
 (see\cite{KLMP}, propositions 2.4.13 и 2.4.22).

Let $X=\{x_1,\dots,x_m\}$ be a finite system of elements. By $T=T(X)$ we
denote the free commutative semigroup with unity (written
multiplicatively), generated by the elements of $X$. Elements of $T$
will be called {\bf monomials}. Let $\theta\in T$,
$\theta=x_1^{e_1}\dots x_m^{e_m}$. By the {\bf order} of $\theta$
we shall call the sum $e_1+\dots+e_m$ that will be denoted by $\ord\theta$. Suppose, that the set of 
monomials
is linearly ordered and for any $\theta\in T$ the following conditions
hold:
$$
  1\leq\theta;
$$
and
if $\theta_1<\theta_2$, then
$$
  \theta\theta_1<\theta\theta_2.
$$
In this case we shall say, that a {\bf ranking} is defined on the set of monomials $T$.

Let $\F$ be a field, $P$ the vector $\F$-space with the basis
$T=T(X)$. We define on $P$ the function ``{\tt taking the leader}'' 
in the
following way: any $g$ in $P$ may be represented as a sum
$g=\sum_{\theta\in T} a_\theta \theta$, where only a finite number of
coefficients $a_\theta\in \F$ are distinct from zero (such representation
is unique up to the order of the terms). Among all monomials, present in
this expression with nonzero coefficients, we choose the maximal with
respect to the order introduced on the set $T$. This monomial will be
called the {\bf leader} of $g\in P$ and will be denoted by $\u_g$. 

\begin{definition}
Let some ranking on the set of monomials $T=T(X)$ be given and let $P$ be
the vector $\F$-space with the basis $T$. Suppose that $P$ is a
$\F$-algebra, and $\u_{A B}=\u_A\u_B$ for all $A,B\in P$. Furthermore,
suppose that $1\theta_1\cdot1\theta_2=1 \theta_1\theta_2\in P$ for any
$\theta_1,\theta_2\in T$; in particular, the generators $x_1,\dots,x_m$
pairwise commute. Such ring we shall call the \bf{ring of generalized
polynomials in the indeterminates} $X=\{x_1,\dots,x_m\}$.
\end{definition}

\begin{example}\label{4.1.10} {\it The ring of commutative polynomials over a field.}
Consider an arbitrary ranking on the set $X=\{x_1,\dots,x_m\}$. Let $P$ be
the algebra $\F[x_1, \dots,x_m]$ of polynomials in the commutative
indeterminates $x_1,\dots,x_m$ over a field $\F$. It is easy to see, that
the condition $\u_{A B}=\u_A\u_B$ holds for all $A,B\in P$ and therefore
we may treat $\F[x_1,\dots,x_m]$ as a ring of generalized polynomials in
the indeterminates $x_1,\dots,x_m$.
\end{example}

\begin{definition} An operator $\partial$ on a 
ring $\K$ is called a
{\bf derivation operator} (or {\bf differentiation}) iff 
$\partial(a+b)=\partial(a)+\partial(b)$ and $\partial(ab)=
a\partial(b)+\partial(a)b$ for all $a,b\in \K$.

A commutative ring $\K$ with a finite set 
$\Delta = \{\partial_1,\ldots,\partial_m\}$ of
mutually commuting derivation operators on $\K$ is called a 
{\bf   differential  ring}. 
\end{definition}

\begin{definition}\label{3.2.38}
Let $\F$ be a differential field and let 
$\Delta = \{\partial_1,\ldots,\partial_m\}$
 be a
basic set of derivation operators on $\F$. The ring
$D=\F[\partial_1,\ldots,\partial_m]$  of skew polynomials in indeterminates $\partial_1,\ldots,\partial_m$ 
with coefficients in $\F$ and the
commutation rules 
$\partial_i \partial_j=\partial_j \partial_i,\ \partial_i a=
a \partial_i+\partial_i(a)$
 for all
$a\in \F,\ \partial_i, \partial_j \in \Delta$
is called a {\bf (linear) differential (or $\Delta$-) operator ring}.
 on $\F$.
\end{definition}
\begin{example}\label{4.1.11}
 {\it The ring of differential operators over a field.}
Let $\F$ be a $\Delta$-field,
and let an arbitrary ranking be fixed on the set 
$T=T(\Delta)$. Then the ring
$D=\F[\partial_1,\dots,\partial_m]$
of linear differential operators over $\F$ (see
definition~\ref{3.2.38}) is a ring of generalized 
polynomials in the indeterminates $\partial_1,\dots,\partial_m$.
\end{example}

\begin{example}\label{4.1.12} {\it The ring of 
differential operators over a ring of
polynomials.}
Let $\F$ be a  $\Delta=\{\partial_1,\dots \partial_m\}$-field 
and $R$ a ring of commutative polynomials in
the indeterminates $y_1,\dots,y_n$ over $\F$. 
We define the derivation
operators $\Delta'=\{\partial'_1,\dots \partial'_m\}$ on $R$ 
in the following way. Set   $\partial'_i(f)=0$ for all
$i=1,\dots,m$,  $j=1,\dots,n$. For any
$1\le i\le m$ we fix a number $1\le j\le n$ (for different $i$ the corresponding indices $j$ may coincide) and set 
$\partial'_i(f)=\partial_i(f)y_j$ for
all $j=1,\dots,n$ and $f\in \F$. 
Then the ring $D_R$ of linear
$\Delta'$-operators over $R$ is a ring of generalized polynomials in
the indeterminates $X=\{\partial'_1,\dots,\partial'_m,y_1,\dots,y_n\}$. 
Indeed, if we
consider the ranking such that $\partial'_i>y_j$ 
for all $i=1,\dots,m$,
$j=1,\dots,n$, then the condition $\u_f\u_g=\u_{f g}$ is 
fulfilled.
\end{example}
Let $D$ be a ring of generalized polynomials in the indeterminates
$X=\{x_1,\dots,x_m\}$ over a field $\F$ and $F$ the free
$D$-module with the basis $B=\{f_1,\dots,f_n\}$. The $\F$-vector space
$F$ has as a basis the direct (Cartesian) product $T\times B$ of the
sets $T=T(X)$ and $B$. This product we shall call the {\bf set of terms} 
of   the module $F$,
$$
T_F=\{x_1^{i_1}\dots x_m^{i_m}f_j\ |\ (i_1,\dots,i_m)\in\N_0^m,\
  j=1,\dots,n\}.
$$
We cannot multiply terms, but we can define the product of a term by a
monomial satisfying the following conditions:

$ 
  \text{for any term}\ u\leq v, \ u,v\in T_F,\text{and for any  monomial }
\ \theta\in T
\ \text{true}\  \theta u\leq \theta v.
$

\begin{definition}\label{4.1.11}
A ranking will be called {\bf orderly} if the condition 
$\ord\theta_1<\ord\theta_2\
(\theta_1,\theta_2\in T)$ implies $\theta_1 f_i<\theta_2 f_j$ for all 
$1\leq  i,j\leq n$.
\end{definition}

\begin{example}\label{4.1.12}
Let a ranking on the set $T$ of monomials be given. 
We shall order the
terms  $T_F$:  $\theta_1f_i<\theta_2f_j$
if either $i<j$ or $i=j$ and  $\theta_1<\theta_2$.
Such ranking on $T_F$ is not orderly.
\end{example}

\begin{example}\label{4.1.13}
Let   a ranking  on the set $T$  is following:
$\theta_1<\theta_2$ iff  either
$\ord\theta_1<\ord\theta_2$, or $\ord\theta_1=\ord\theta_2$
and $\theta_1<\theta_2$   with respect to lexicographic order 
on monomials.
Let
$t_1,t_2\in T_F$. 
We set $t_1=\theta_1f_i<t_2=\theta_2f_j$ if
and only if
 either 
$\ord\theta_1<\ord\theta_2$,
or $\ord\theta_1=\ord\theta_2$ and $i<j$,
or $\ord\theta_1=\ord\theta_2$, 
$i=j$ and $\theta_1<\theta_2$.
This ranking is orderly. We shall call it {\bf standard}.
\end{example}

In the submodule of the free module $F$ over the ring of generalized
polynomials $D$
  a Gr\"obner basis exists:
\begin{definition} (see definition 4.1.25, \cite{KLMP}).
Let $D$ be a ring of generalized polynomials in indeterminates
$X=\{x_1,\dots,x_m\}$, $F$ a free $D$-module. Suppose that $M
\subseteq F$ is a submodule of $F$, $G \subset M$ is a finite set 
and $<$
is a ranking on the set of terms $T_F$. The set $G$ is called a 
{\bf Gr\"obner
basis} of $M$, if there exists for any nonzero 
$f\in M$ a
 representation:
$$
  f=\sum_{i=1}^r c_i\theta_ig_i,\quad  0\ne c_i\in \F,\  \theta_i
       \in  T(X),\ g_i  \in  G,\\
  \theta_i\u_{g_i}> \theta_{i+1}  \u_{g_{i+1}},
$$
that, in particular, implies
$$
  \u_f = \theta _1 \u_{g_1}.
$$
\end{definition}
 
We shall now consider graded modules over the ring of generalized
polynomials. Firstly  we consider in $T=T(X)$ the subset
$$
 T_s=\left\{x_1^{i_1}\dots x_m^{i_m}| \sum_{k=1}^m i_k=s,\ (i_1,\dots,i_m)\in\N_0^m \right\},
$$
$s\in\Z$ and  $T_s=\empty$ for all $s<0$.

\begin{definition}
Let $D$ be a ring of generalized polynomials over a field $K$ in
the indeterminates $X=\{x_1,\dots,x_m\}$
 We suppose the ranking of $T=T(X)$ to be orderly, and
$$
 D_s=\Bigl\{\sum_{\theta\in T_s} a_\theta\theta\ |\
a_\theta\in \F\ 
\vtop{\hbox{
and almost all coefficients } }\hbox{
are equal to 0.}  
\Bigr\}
$$

The ring $D$ will be called {\bf graded} if
$$
  D=\bigoplus_{s\in\N_0} D_{s}\qquad\text{and}\qquad
  D_{s} D_{r}\subseteq D_{s+r}
  $$
for all $s,r\in\N_0$.
\end{definition}

The rings (examples~\ref{4.1.10},~\ref{4.1.12}), 
with standart ranking
(see example~\ref{4.1.13}),  are graded.
The ring of differential operators over field
$\F$ (example~\ref{4.1.11}) 
isn't graded, if
there are non-constant elements in field
 $\F$.

\begin{definition}
Let $D$ be a graded ring of generalized polynomials over a field 
$\F$.
A $D$-module $M$ will be called {\bf graded}, if for any 
$s\in\N_0$     a $\F$-subspace $M_{s}$ of $M$ is defined such
 $M=\bigoplus_{s\in\N_0} M_{s}$ and
$$
  D_{s} M_{r}\subseteq M_{s+r}
$$
  
for all $s,\ r\in\N_0$. The elements of $M_{s}$ will
be called the {\bf homogeneous elements} of degree $s$.
\end{definition}

\begin{definition}
Let $_D M$ be a finitely generated module over a ring of generalized
polynomials and $M=\bigoplus_{s\in\N_0}M_{s}$ be
a grading of $M$. The function $\phi_M^{gr}$, whose value at
any  
$s\in\N_0$ is equal to 
$\dim_\F M_{s}$ will
be called the {\bf characteristic function} of the graded module $M$.
\end{definition}

\begin{theorem}\label{4.3.20} (see \cite{KLMP}, theorem 4.3.20.)
Let $D$ be a graded ring of generalized polynomials over a field in
the indeterminates $X=\{x_1,\dots,x_m\}$,
$_D M$ be a graded module and  $\{m_1,\dots,m_n\}$ be a finite
set of its generators such that $m_i\in M_{(\alpha)^i}$.
Then there exist sets $E_i\subset\N_0^m$ $(i=1,\dots,n)$ such that for all large
enough $s$ the characteristic function of $M$ is equal
to
\begin{equation}\label{eq0}
  \phi_M^{gr}(s)=\Delta_1
  \sum_{i=1}^n\omega_{E_i}(s-\alpha^i),
\end{equation}
where $\omega_E(s)$ 
is the Kolchin dimension
polynomial of the matrix $E$
(see theorem~\ref{Kolchin}, equation~\ref{delta}).
\end{theorem}

As follows from the proof, the sets $E_i$ correspond to
leaders of a homogeneous Gr\"obner basis of 
relations between generators (syzygy module).    It is easy to see that
for sufficiently large s the function  $\phi_M^{gr}(s)$
is polynomial.

We denote it by $\omega_M(s)$ and call {\bf the characteristic polynomial}
graded finitely generated module $_DM $.
Its degree $d(M) =\deg(\omega_M)$ is called (generalized)
{\bf type}  of
module $M$, the difference $(m-1-d)$ - (generalized) {\bf codimension},
and the standard leading (nonzero) coefficient $\tau_{d}(M) $ - (generalized)
{\bf type dimension.}
  
Graded modules over a ring of generalized polynomials
have properties similar to properties
differential modules: $d(M)<m$ and $a_{m-1} (\omega_M) = \rk_DM $.

Let $F$ be a free $D$-module with
generators
$ f_1, \dots, f_n $.
Each element of $f \in F$ is represented as
$f =\sum_ {1 \le j \le n} \theta_jf_j $, where $\theta_j \in D$.
Denote by
$\ord_{f_i} f = \ord \theta_i $ and $\ord f = \max_ {1 \le i \le n}
 (\ord\theta_i) $.

Consider the following grading on $F$:
$ F_s = \sum_ {i = 1} ^ n T_ {s} f_i $.
Let $H$ be the submodule of the module $F$ generated by 
elements of $\Sigma \subset H $,
and $\ord_{f_j} h \le e_j $ for all $j = 1, \dots, n, \ h \in \Sigma$.
The induced grading arises on the module $ H $:
$H_s = H \cap F_s $.
The factor module $F/H$ can also be regarded as graded:
$(F/H)_s = (F_s/H_s)$, and $\omega_ {F / H} (s) = \omega_F (s) - 
\omega_H (s) =
\sum_ {i = 1} ^ n \binom {s + m-1} {m-1} - \omega_M (s) $.
Sometimes this polynomial
called the characteristic polynomial of a system of generalized 
polynomial   equations $\Sigma $ (or a system of $D$ -equations) and
denoted by $\omega_{[\Sigma]} $.

As follows from the theorem~\ref{4.3.20} (put 
$a_i = 0\ (i = 1, \dots, n), \ M = F / H$),
characteristic polynomial of a system of generalized polynomial equations
can be calculated as in the differential case:
(Theorem 4.3.5 \cite {KLMP}):
\begin{equation}\label{eq1}
\omega_{[\Sigma]} (s) = \sum_ {j = 1}^n\Delta_1 \omega_ {E_j} (s),
\end{equation} 
where $E_j\subset \N_0^ m$.

We are interested in following  
\begin{question}\label{1}
How to estimate the typical dimension of $\Sigma $ in
known orders $e_1, \dots, e_n$?
\end{question}

Firstly this question was asked in differential algebra
by J.Ritt
for  ordinary differential systems. Later
E.Kolchin decided this problem in a codimension  for nonlinear
systems.
His bound (see~\cite {Kolchin}, p. 199) is as follows:
typical differential dimension $ a_ {m-1} $ of the system $\Sigma$
does not exceed
$e_1 + \dots + e_n$.
 
In codimension 2, such a result is known  
(see 5.6.7, \cite{KLMP}):

Let $n=1$, then
$a_{m-2}(\omega_{\Sigma})\le e_1^2$. 

Both of these results are also true for systems of homogeneous 
generalized polynomial equations.

\section{Basic results.}

So, for systems of generalized homogeneous polynomial equations
in codimensions 1 and 2, the classical B\'ezout theorem holds.
If the codimension is greater than 2, in the general case this is 
not true.    Consider an example.

\begin{example}\label{ex}
Let    $k\in N$,
$
  \F=\C(x_1,x_2,x_3),\quad
  n=1,\quad D=\F\{\partial_1,\partial_2,\partial_3,y_1\}$ -- 

ring of differential operators over
ring of polynomials in one variable $y_1$,
(see the example~\ref {4.1.12}) over the field $\F $,
  $m=4$,
$$
  \Sigma=\{\partial_1^kf_1,(\partial_2^k+x_1\partial_3^k)f_1\}.
$$
Let's we have the orderly rank $\partial_1>\partial_2> \partial_3> y_1 $.
One can find a homogeneous Gr\"obner basis of the ideal $[\Sigma]$. 
It consists of
elements

$G=\{\partial_1^kf_1$, $(\partial_2^k+x_1\partial_3^k)f_1$, 
$\partial_1^{k-1}\partial_3^{k}y_1f_1$,
$\partial_1^{k-2}\partial_3^{2k}y_1^2f_1$, ...,
$\partial_1^{k-i}\partial_3^{i*k}y_1^if_1$, ...,
$\partial_3^{k^2}y_1^kf_1\}$,                                                                                                                               

From here according to the equation~(\ref{eq1}), 
$\omega_{[\Sigma ]}=\Delta_1\omega_E$, where 
$$
E = \left( 
\begin{matrix} 
k & 0 & 0 & 0\\ 
0 & k & 0 & 0\\
k-1&0&k &1\\
\dots&\dots&\dots&\dots\\
k-i&0&ik &i\\
\dots&\dots&\dots&\dots\\
0 & 0 & k^2&k\\
\end{matrix} 
\right).
$$ 

One of the main methods for calculating the dimension polynomial of a matrix
is the using of the formula (see \cite {KLMP}, Theorem 2.2.10):

\begin{equation}\label{eq2}
\omega_E(s)=\omega_{E\cup e}+\omega_H(s-\ord (e)),
\end{equation}
where $e\in N_0^m$, $H$  is the matrix obtained by subtracting the 
vector $e$ from each row of $E$ (negative numbers are replaced by zeros).

Apply the formula~(\ref {eq2}) 
$e =(0, k, 0,0) $ $k$ times, we obtain 
$\omega_E=k\omega_{E_1}$, here 
$$
E_1 = \left( 
\begin{matrix} 
k &  0 & 0\\ 
k-1&k &1\\
\dots&\dots&\dots\\
k-i&ik &i\\
\dots&\dots&\dots\\
0 &  k^2&k\\
\end{matrix} 
\right).
$$ 
 
By Theorem 2.2.17 (see \cite {KLMP}) we have:
  
$\Delta_1(\omega_{E_1})=\omega_{E_2}+
\omega_{E_3}$, where $E_2$, $E_3$ are the matrixs, 
obtained by 
deletion, respectively,
second and third columns of the matrix $E_1$.
Applying corollary 2.3.21 (see~\cite {KLMP}),
we get 
 $\omega_{E_2}=1+2+
\dots+k=k(k+1)/2$ and 
$\omega_{E_3}=k(1+2+\dots+k)=k^2(k+1)/2$,
whence  $\omega_E=k^2(k+1)^2/2$.
If the classical B\'ezout theorem holds for the system $\Sigma$,
we would have to have $\omega_{[\Sigma]} = k^2 (k + 1)^2/2 \leq k^3$
(the system has codimension 3, while it has 2 homogeneous generators),
which is wrong.
\end{example}

\begin{theorem}\label{prim}
Let $D$ be a graded
ring of generalized polynomials over the field $\F$ in indeterminates
$X =\{x_1, \dots, x_m \}$,
$F = \bigoplus_{i = 1}^nD $ - free graded
$D$-module with generators $f_1, \dots, f_n $,
$\Sigma \subset F $ is a system of   homogeneous
 $D$-equations.   Let be   
$\ord_ {f_j} h \le e_j $ for any $h \in \Sigma $.

Then the following statements are true:

if the codimension of the system is 0, then
typical differential dimension 
 does not exceed $n$;

if the codimension of the system $ \Sigma $ is 1, then
$ \tau_d (\Sigma) \le e_1 + \ dots + e_n $;

if the codimension of $ d(\Sigma) $ is 2, then
$\tau_d (\Sigma) \le (e_1 + \dots + e_n)
\max_ {i = 1} ^ ne_i + \ prod_ {i <j} e_ie_j 
\le (e_1 + \dots e_n)^2$.

This bound is being achieved, see an example from~\cite{Kondr1}.
\end{theorem}

First we prove the lemma.
\begin{lemma}\label{lemma1}
Suppose that under the conditions of the theorem~\ref{prim} 
the generalized type of the system $ \Sigma $ is greater than 1.
Then

\begin {equation} \label {W}
\omega_\Sigma (s) = \sum_ {i = 1}^n
\left (\binom {s + m-1} {m-1} - \binom {s + m-1- \tilde e_i}{m-1}\right)
-w (s),
\end {equation}

where $w(s + e)\in W $, $e=\max_ {i = 1}^ n (e_i) $, $ \tilde e_i 
\le e_i$.
\end{lemma}
\begin{proof}
We denote by $H$ the submodule of the $D$-module $F$ 
generated by the system   $\Sigma$.
$D$ is an Ore ring and, since the codimension of $\Sigma$ is greater 
than 1,  $\rk_D F/H = 0$, whence
$\rk_D H = n$. We choose $n$ $D$-independent equations from
 $\Sigma $,
and let $ _DM$ - $D$-factor module $F$ by
these equations (we denote them by $\Sigma_1$. We have the exact 
sequence of graded $D$-modules:

\begin{equation}\label{toch}
0\to N\to M->\to F/H\to 0.
\end{equation}

We can assume that $N$ is a graded submodule of the module $M$,
generated by the equations $\Sigma\backslash\Sigma_1$, and let 
$\alpha^i$ - the degrees of these generators.
Gradings, associated with the choice of  homogeneous 
generators ($N_s=D_{s- \alpha^i} g_i$) and grading,
induced by $M$
($N_s = M_s \cap N$) 
coincide.

Let $e =\max (e_1, \dots, e_n)$. From the theorem~\ref{4.3.20} it 
follows that
$\omega_N (s + e)\in W $. Indeed, the generators 
$g_i$ of the module $N$
have degrees $\alpha^i$ not greater than $e_i$ and from the 
formula~(\ref{eq0})
we get that $\omega_N(s + e)=\sum_{i=1}^k\omega_i(s+e-\alpha^i)$, 
where   $\omega_i \in W$
(here  $k = \Card (\Sigma) -n)$. Because $\alpha^i\le e_i$,
keeping in mind  closed $W$ relatively positive
shift and summation  we get that $\omega_N(s + e)\in W$.

To calculate $\omega_M = \omega_{\Sigma_1}$ we use
$D$-independence of the equations $\Sigma_1$. Proof of Lemma 5.8.2
(\cite {KLMP}) can also be done for a system of generalized polynomial
equations, therefore we have: $J(\Sigma_1) \ne \infty $, where $J$ is 
the Jacobi number  of
matrix $(\ord_ {f_i} h_j)_{i, j = 1}^n$, $h_i\in \Sigma_1$. 
Choosing the final diagonal
 sum of the matrix and renumbering the equations $\Sigma_1$, 
since $\ord h_ {f_i} h_i \ne \infty $,
we get
\begin{align}
&\omega_{\Sigma_1}(s)=\omega_F(s)-\sum_{i=1}^n\binom{s+m-1-\ord h_i}{m-1}= \notag  \\
&
\sum_{i=1}^n\binom{s+m-1}{m-1}-\sum_{i=1}^n\binom{s+m-1-\tilde e_i}{m-1}, \notag
\end{align}
where $\tilde e_i=\ord_{f_i} h_i\le e_i$. 

To prove the lemma it remains to use the equality
$\omega_M =\omega_N + \omega_{F / H}$,
obtained from the sequence~(\ref{toch}).

\end{proof}
We return to the proof of the theorem.
\begin{proof}
The case $d(\Sigma) = m-1 $ follows  from
Theorems~\ref{4.3.20}.

Let $d (\Sigma) = m-2 $.
It follows from the lemma~\ref{W} that
$\sum_ {i = 1}^n((s + 1) - (s + 1- \tilde e_i)) - \Delta_1 ^ {m-2}
\omega_\Sigma (s) \in W $ (since $W$ is closed relative to the operation
$\Delta_1$, see the formula~(\ref{delta})). From here
$\sum_{i = 1}^ n\tilde e_i- \tau_d (\Sigma) \in W $. 
Because minimizing coefficients of a polynomial from $W$ are 
non-negative,  we immediately get that
$\tau_d (\Sigma) \le \sum_ {i = 1} ^ n \tilde e_i \le \sum_ {i = 1}^ne_i $.

Let $d(\Sigma)=m-3$. As above, we use the operator $\Delta_1^{m-3} $
to expression~(\ref{W}). We get:

\begin{align}
&\Delta_1^{m-3}\left(\sum_{i=1}^n
\binom{s+m-1}{m-1}-\binom{s+m-1-\tilde e_i}{m-1}\right)
 -\tau_d(\Sigma)=w'(s),\notag  \\
& w'(s+e)\in W,  \notag  
\end{align}
whence
$$
\left(\sum_{i=1}^n
\binom{s+2}{2}-\binom{s+2-\tilde e_i}{2}\right)
 -\tau_d(\Sigma)=w'(s),\ w'(s+e)\in W
$$
and $\sum_{i=1}^n (\tilde e_i(s+1)-\binom{\tilde e_i}{2})-\tau_d(\Sigma)=w'(s)$.
Let the minimizing coefficients of the polynomial $w'(s + e)$ be 
equal to $(b_1, b_0)$.
Then $w '(s + e) = b_1 (s + 1)-\binom{b_1}{2} + b_2 $, 
and $ b_1> = 0, \ b_2> = 0 $.
We have:

$\tau_d(\Sigma)=(\sum_{i=1}^n\tilde e_i-b_1)(s+1)-\sum_{i=1}^n(\tilde e_i-\binom{\tilde e_i}{2})+eb_1+\binom{b_1}{2}-b_0
$.

Equating the coefficient in $s$ to the right side of the 
equality to zero, we obtain
$b_1=\sum_{i=1}^n\tilde e_i$, whence $\tau_d(\Sigma)\le(\sum_{i=1}^n\tilde e_i)
e-\sum_{i=1}^n\binom{\tilde e_i}{2}+\binom{\sum_{i=1}^n\tilde e_i}{2}=\prod_{i<j}\tilde e_i\tilde e_j
+(\sum_{i=1}^n\tilde e_i)e\le(\sum_{i=1}^ne_i)\max_{i=1}^ne_e+\prod_{i,j=1,iБ j}^ne_ie_j$
\end{proof}

Further we consider homogeneous
ideals in the ring of generalized polynomials, i.e. case $n = 1$.
Let an ideal be generated by elements of order no higher than $ e $.
If the generalized type of an ideal is 2, then from the theorem~\ref{prim}
follows that
its typical dimension does not exceed $e^2$, i.e. the classical B\'ezout theorem holds.

\begin{theorem}\label{prim2}
Let $D$ be a graded
ring of generalized polynomials over the field $\F$ in indeterminates
$X = \{x_1, \dots, x_m \}$,
$\Sigma \subset D$ is a system of homogeneous $D$-equations.
Let be
$\ord h \le e$ for any
$h\in\Sigma$.

Then the following bounds are true:

if the codimension of $\Sigma$ is 3, then
$\tau_d(\Sigma) \le e^2 (e + 1)^2/2$ (according to the example~\ref{ex}
this estimate is achieved);

if the codimension of $\Sigma$ is 4, then
$\tau_d(\Sigma)\le e^2(e+1)^2(3e^4+6e^3+11e^2+8e+8)/24$;

if the codimension of $\Sigma$ is 5, then
$\tau_d(\Sigma)\le 
  e^2 (e+1)^2 (288 +  480 e + 952 e^2  + 1264 e^3  + 1592 e^4
                          + 1648 e^5  + 1529 e^6  + 1174 e^7+ 775 e^8  + 420 e^9  + 183 e^10   + 54 e^11
         + 9 e^12  ) (e + 1)^2 /1152 
$

in any codimension $\tau> 0$ the generalized typical dimension 
$\tau_d(\Sigma)$
does not exceed $O(e^{2^{\tau-1}})$.
\end{theorem}  
\begin{proof}
is 
based on the lemma~\ref{lemma1} and the fact that minimizing 
coefficients are multivalued from the set $W$ are non-negative.
For $\Sigma'$, we choose the element $\Sigma$ in maximal order,
let it be $e$. Then, in the equation~(\ref{W}) $n = 1$, 
$\tilde e_1 = e $.

Consider the case of codimension 3.
We apply the operator $\Delta_1^{m-4}$ to both 
sides of the equality~(\ref{W}).

$$
\tau_d(\Sigma)=
\left(
\binom{s+3}{3}-\binom{s+3-e}{3}\right)
 -w'(s),\ w'(s+e)\in W
$$

Let the sequence of minimizing coefficients of the polynomial $w'$
is $(b_2, b_1, b_0)$. According to the definition~\ref{min}, we can 
explicitly express the standard coefficients
$w'$ through the numbers $b_2$, $b_1$, $b_0$ and find the coefficients 
of 'shifted'
the polynomial $w'(s + e)$.
Equating the coefficients at $s^2$, $s$
on the right side of the equation to 0, we get:
$b_2 = e$, $b_1 =e^2$ and
\begin{align}
&
\tau_d(\Sigma)=
\left(\binom{s+3}{3}-\binom{s+3-e}{3}\right) -
\left(\binom{s+3-e}{3}-\binom{s+3-2e}{3}\right) -  \notag  \\
&
\left(\binom{s+2-2e}{2}-\binom{s+2-2e-b_1}{2}\right) - b_0.  \notag
\end{align}

Substituting $s=-1$, we get $\tau_d(\Sigma) \le e^2 (e + 1)^2/2 $.

Bounds in any codimension are calculated in the same way.
Each time we will receive a polynomial in $e$.

If the precise coefficients of this polynomial are not important, 
but only its degree in $e$,
it is claimed to be $2^{\tau-1} $. Indeed,
let $d$ be the generalized dimension of the system $\Sigma$, i.e.
codimension $\tau$ of
the polynomial $\omega_\Sigma $ is equal to $\tau = m-1-d$.
Apply to~(\ref{W}) operator
$\Delta_1^{d}$ (while
$\Delta_1^{d}\omega_\Sigma$ is a polynomial of degree zero,
i.e.  constant = $\tau_d(\Sigma)$).
Comparing the degrees, we get that the degree $w '=\Delta_1^{d} w $ 
is less than $\tau $.
Let the minimizing coefficients of the polynomial $w '$
equal to $ (0,\dots, 0, b _{\tau-1}, \dots, b_0)$.
Replace in the resulting equation
$s$ variable
on $e$ and we have the following relation:
$$
\tau_d(\Sigma)=\binom{s+\tau+e}{\tau}-\binom{s+\tau}{\tau}-w'(s).
$$

We use the definition~(\ref{min}) and get
\begin{align}\label{fin}
&
\tau_d(\Sigma)=\binom{s+\tau+e}{\tau}-\binom{s+\tau}{\tau}-   \\ 
&
\sum_{k=\tau}^{1}
\left(\binom{s+k-\sum_{j=\tau}^{k}b_j}{k}-
\binom{s+k-\sum_{j=\tau}^{k-1}b_j}{k}\right).  \notag
\end{align}

Denote by $c_i= \sum_{j = i} ^ {\tau-1} b_j$
and rewrite
equation~(\ref{fin})
in this form:
\begin{align}\label{fin1}
&\tau_d(\Sigma)=\binom{s+\tau+e}{\tau}-\binom{s+\tau}{\tau}-   \\
&\sum_{k=\tau}^{1}
\left(\binom{s+k-c_k
}{k}-\binom{s+k-c_{k-1}
}{k}\right), \notag
\end{align}

Using identity
$$
\binom{s+k-1-a}{k}=\binom{s+k-a}{k}-\binom{s+k-1-a}{k-1},
$$
transform the equation~(\ref{fin1}) to the form:
\begin{equation}\label{fin2} 
\tau_d(\Sigma)=\binom{s+\tau+e}{\tau}-2\binom{s+\tau}{\tau}+   
\sum_{k=\tau}^{2}
\binom{s+k-1-c_{k-1}}{k}
+(s+1-c_0).
\end{equation}
Take $\Delta_1^{\tau-1}$ from both sides of the equality~(\ref{fin2}).
Will have:
$0=(s+1+e)-2(s+1)+(s-c_{\tau-1})+1$, откуда $c_{\tau-1}=e$.
By induction on $i$, we prove that for $1\le i <\tau-1$
it holds: 
$$c_i=O(e^{2^{(\tau-1)-i}}).$$
Let  $c_j=O(e^{2^{(\tau-1)-j}})$ 
for all
 $j:\ i\le j<\tau-1$.

Take $\Delta_1^{i-1}$ from both sides of the equality~(\ref{fin2}) 
and get:
$$
0
=\binom{s+\tau-i+1+e}{\tau-i+1}-2\binom{s+\tau-i+1}{\tau-i+1}+\sum_{k=\tau}^{i-1}\binom{s+k-i-c_{k-1}}{k-i+1}
$$

Substituting $-1$ instead of $s$,
we get
$$
0=\binom{\tau-i+e}{\tau-i+1}-2\binom{\tau-i}{\tau-i+1}+\sum_{k=\tau}^{i+1}\binom{k-i-1-c_{k-1}}{k-i+1} -c_{i-1}
$$

In the last sum we make the change $j= k-i + 1$ and get
$$
0=\binom{\tau-i+e}{\tau-i+1}+\sum^{j=\tau+1-i}_{j=2}\binom{j-c_{i+j-2}}{j} -c_{i-1}.
$$

Now we have a formula expressing
$c_{i-1}$ through $c_{i}, ..., c_{\tau-1}$:
\begin{align}
& 
c_{i-1}= \sum_{j=2}^{\tau-i+1}O\left(\binom{c_{i+j-2}}{j}\right) = 
\sum_{j=2}^{\tau-i+1}O(c_{i+j-2}^j)= 
\sum_{j=2}^{\tau-i+1}O(e^{j2^{\tau-i-j+1}})=
\notag\\
&
O(e^{2\cdot2^{\tau-1-i}})+ \sum_{j=3}^{\tau-i+1}O(e^{j2^{\tau-i-j+1}}) = 
O(e^{2^{\tau-i}})+\sum_{j=3}^{\tau-i+1}O(e^{2^{j-1}\cdot2^{\tau-i-j+1}})  = \notag \\
& 
O(e^{2^{\tau-i}})+\sum_{j=3}^{\tau-i}O(e^{2^{\tau-i}})= 
O(e^{2^{\tau-i}})
\notag
\end{align}

(we used the inductive assumption
and the fact that $j\le2^{j-1}$ for all $j \ge 2$).

Substituting $s = -1$ into the equation~(\ref{fin2}), we obtain
$
\tau_d(\Sigma)=O(c_1^2)-c_0\le O(2^{\tau-1})$.
\end{proof}

It is not known whether the resulting
double exponential typical dimension bound of
graded ideal in a ring of generalized polynomials is being achieved.
Note that it was proved in (\cite{Mayr}) that for
degrees of
elements in the Gr\"obner basis of the polynomial ideal
double exponential bound is achieved from below.

\end{document}